\documentclass[10pt, a4paper]{article}
\usepackage[utf8]{inputenc}

\usepackage[T1]{fontenc} % codage des fontes TeX
\usepackage[english]{babel}

\usepackage[usenames,dvipsnames]{pstricks}
\usepackage{epsfig}
\usepackage{pst-grad} % For gradients
\usepackage{pst-plot} % For axes
\usepackage{euscript}
\usepackage{amsmath}
\usepackage{amsfonts}
\usepackage{amssymb}
\usepackage{float}
\usepackage{graphicx}

\usepackage{xcolor}

\newcommand{\skipline}{\vspace{0.6cm}}

\newcommand{\N}{\mathbb{N}}
\newcommand{\R}{\mathbb{R}}

\usepackage{amsthm}
\theoremstyle{plain}
\newtheorem{theorem}{Theorem}
\newtheorem{lemma}{Lemma}
\newtheorem{remark}{Remark}
\newtheorem{open}{Open problem}

\newcommand{\EnsU}{L^\infty(0,T)}
\newcommand{\EnsV}{H^{1/4}(0,T)}
\newcommand{\diverg}{\mathop{\text{div}}\nolimits}

\title{Fast global null controllability for a viscous Burgers' equation 
despite the presence of a boundary layer}
\author{Frédéric Marbach
\footnote{Email: marbach@ann.jussieu.fr. Address: 
Laboratoire Jacques-Louis Lions, Université Pierre et Marie Curie, 
Institut Universitaire de France, 4, Place Jussieu, 75252 Paris Cedex, France.
Work partially supported by the ERC advanced grant 266907 (CPDENL) of the 
7th Research Framework Programme (FP7)}}

\begin{document}

\maketitle

\begin{abstract}
In this work, we are interested in the small time global null controllability 
for the viscous Burgers' equation $y_t - y_{xx} + y y_x = u(t)$ on the line 
segment $[0,1]$. The second-hand side is a scalar control playing a role similar 
to that of a pressure. We set $y(t,1) = 0$ and restrict ourselves to using only 
two controls (namely the interior one $u(t)$ and the boundary one $y(t,0)$).
In this setting, we show that small time global null controllability still holds 
by taking advantage of both hyperbolic and parabolic behaviors of our system. We 
use the Cole-Hopf transform and Fourier series to derive precise estimates for 
the creation and the dissipation of a boundary layer.
\end{abstract}

% ==============================================================================
\section{Introduction}
% ==============================================================================

% ==============================================================================
\subsection{Description of the system and our main result}

Let $T > 0$ be a positive time, possibly small. We consider the line segment 
$x \in [0,1]$ and the following one-dimensional viscous Burgers' controlled 
system:
\begin{equation}
	\label{burgers2}
	\left\{
	\begin{array}{rll}
		y_t + y y_x - y_{xx} & = u(t)
		& \quad \mbox{in} \quad (0,T) \times (0,1), \\
		y(t, 0) & = v(t) &	\quad \mbox{in} \quad (0, T), \\
		y(t, 1) & = 0 &	\quad \mbox{in} \quad (0, T), \\
		y(0, x) & = y_0(x) & \quad \mbox{in} \quad (0,1).
	\end{array}
	\right.
\end{equation}
The scalar controls are $u \in L^2(0,T)$ and $v \in H^{1/4}(0,T)$. The 
second-hand side control term $u(\cdot)$ plays a role somewhat similar
to that of a pressure for multi-dimensional fluid systems. Unlike some other 
studies, our control term $u(\cdot)$ depends only on time and not on the space 
variable.

For any initial data $y_0 \in L^2(0,1)$ and any fixed controls in the 
appropriate spaces, it can be shown that system~(\ref{burgers2}) has a 
unique solution in the space 
$X = L^2((0,T); H^1(0,1)) \cap \mathcal{C}^0([0,T]; L^2(0,1))$.
This type of existence result relies on standard \textit{a priori} estimates
and the use of a fixed point theorem. Such techniques are described in
\cite{MR0259693}. One can also use a semi-group method as in~\cite{MR710486}.
Our main result is the following small time global null controllability theorem 
for system~(\ref{burgers2}):

\begin{theorem}
\label{thm.frederic}
Let $T > 0$ be any positive time and $y_0$ by any initial data in $L^2(0,1)$.
Then there exists a control pair $u \in \EnsU$ and $v \in \EnsV$ such that the
solution $y \in X$ to system~(\ref{burgers2}) is null at time $T$. That is to 
say, $y$ is such that $y(T, \cdot) \equiv 0$.
\end{theorem}

% ==============================================================================
\subsection{An open-problem for Navier-Stokes as a motivation}

As a motivation for our study, let us introduce the following challenging open
problem. Take some smooth connected bounded domain $\Omega$ in $\R^2$ or $\R^3$.
Consider some open part $\Gamma$ of its boundary $\partial \Omega$. This is the
part of the boundary on which our control will act. We consider the following
Navier-Stokes system: 
\begin{equation}
	\label{navier.stokes}
	\left\{
	\begin{array}{rll}
		y_t - \Delta y + (y\cdot\nabla) y & = - \nabla p
		& \quad \mbox{in} \quad (0,T) \times \Omega, \\
		\diverg y & = 0 & \quad \mbox{in} \quad (0,T) \times \Omega, \\
		y & = 0  & \quad \mbox{on} 
		\quad (0,T) \times (\partial \Omega \setminus \Gamma) , \\
		y(0, \cdot) & = y_0(\cdot) & \quad \mbox{in} \quad \Omega.
	\end{array}
	\right.
\end{equation}
We consider this system as an underdetermined system. Our control will be some
appropriate trace of a solution on the controlled boundary $\Gamma$. 

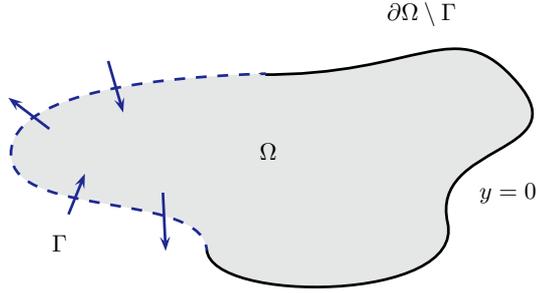
\begin{figure}[!ht]
\centering
% Generated with LaTeXDraw 2.0.8
% Tue Dec 18 13:38:44 CET 2012
% \usepackage[usenames,dvipsnames]{pstricks}
% \usepackage{epsfig}
% \usepackage{pst-grad} % For gradients
% \usepackage{pst-plot} % For axes
\scalebox{0.9} % Change this value to rescale the drawing.
{
\begin{pspicture}(0,-2.236719)(9.529062,2.2567186)
\definecolor{gris}{rgb}{0.8980392156862745,0.9058823529411765,0.9058823529411765}
\definecolor{bleu}{rgb}{0.11764705882352941,0.1607843137254902,0.5725490196078431}
\psbezier[linewidth=0.04,fillstyle=solid,fillcolor=gris](3.0554688,-1.4167187)(3.1154687,-2.2167187)(6.835469,-2.1967187)(6.5754685,-1.0167187)(6.315469,0.16328125)(8.4982395,-0.021577183)(7.595469,1.0432812)(6.6926975,2.1081395)(6.2,1.1367188)(3.8154688,1.1632812)
\psbezier[linewidth=0.04,linecolor=bleu,fillstyle=solid,fillcolor=gris,linestyle=dashed,dash=0.16cm 0.16cm](3.0554688,-1.4167187)(2.9754689,-0.5367187)(0.14229742,-0.9753041)(0.19546875,0.0232812)(0.2486401,1.0218666)(2.9354687,1.1432812)(3.92,1.1767187)
\psline[linewidth=0.04cm,linecolor=bleu,arrowsize=0.073cm 2.0,arrowlength=1.4,arrowinset=0.4]{->}(1.6154687,1.3632811)(1.8354688,0.6032812)
\psline[linewidth=0.04cm,linecolor=bleu,arrowsize=0.073cm 2.0,arrowlength=1.4,arrowinset=0.4]{->}(1.0354688,-0.8967188)(1.2754687,-0.2967188)
\psline[linewidth=0.04cm,linecolor=bleu,arrowsize=0.073cm 2.0,arrowlength=1.4,arrowinset=0.4]{->}(0.7554688,0.3632812)(0.15546876,0.82328117)
\psline[linewidth=0.04cm,linecolor=bleu,arrowsize=0.073cm 2.0,arrowlength=1.4,arrowinset=0.4]{->}(2.4154687,-0.5767188)(2.4554687,-1.4367187)
\usefont{T1}{ptm}{m}{n}
\rput(3.9545312,0.0332812){$\Omega$}
\usefont{T1}{ptm}{m}{n}
\rput(6.1945314,2.0532813){$\partial \Omega \setminus \Gamma$}
\usefont{T1}{ptm}{m}{n}
\rput(0.90453124,-1.3267188){$\Gamma$}
\usefont{T1}{ptm}{m}{n}
\rput(7.454531,-0.5867188){$y = 0$}
\end{pspicture} 
}
\caption{Setting of the Navier-Stokes control problem~(\ref{navier.stokes}).}
\end{figure}

\begin{open}
\label{open.NS}
Is system~(\ref{navier.stokes}) small time globally null controllable? That is 
to say, for any $T > 0$ and $y_0$ in some appropriate space, does there exist a 
trajectory of system~(\ref{navier.stokes}) such that $y(T, \cdot) \equiv 0$?
\end{open}

Many works have be done in this direction. Generally speaking, one can 
distinguish two approaches. First, one can think of the nonlinear term as a 
perturbation term and obtain the controllability by means of the Laplacian term. 
For instance, Fabre uses in~\cite{MR1418484} a truncation method for the
Navier-Stokes equation. In~\cite{MR1648554}, Lions and Zuazua use Galerkin
approximations for various fluid systems. 
Of course, this approach is very efficient for local results.
The most recent result concerning local controllability for 
system~(\ref{navier.stokes}) is the one contained in~\cite{MR2103189} by 
Fern\'andez-Cara, Guerrero, Imanuvilov and Puel. Their proof uses Carleman
estimates. 

The other approach goes the other way around. Indeed, in finite dimension, it 
is known that if $\dot{y} = F(y) + Bu$ where $F$ is quadratic is controllable, 
then $\dot{y} = F(y) + Ay + Bu$ is controllable too 
(see~\cite[Theorem 3.8]{MR2827892}). Likewise, for fluid systems, trying to get 
a fast controllability result implies to work at high Reynolds number (ie. with 
big fluid velocities, or low viscosity) inside the domain. Therefore, inertial 
forces prevail and the fluid system behaves like its null viscosity hyperbolic 
limit system. In our case, we expect to deduce results for Navier-Stokes from 
the Euler sytem. For Euler, global controllability has been shown 
in~\cite{MR1233425} by Coron for the 2D case (see also~\cite{MR1380673}) and
by Glass for the 3D case in~\cite{MR1745685}. Their proofs rely on the return 
method introduced by Coron in~\cite{MR1164379} (see 
also~\cite[Chapter 6]{MR2302744}). For Navier-Stokes, things get harder. 
In~\cite{MR1470445}, Coron and Fursikov show a global controllability result in 
the case of a 2D manifold without boundary. In~\cite{MR1728643}, Fursikov and 
Imanuvilov show a global exact controllability result for 3D Navier-Stokes with 
a control acting on the whole boundary (ie. $\Gamma = \partial \Omega$).

Other approaches exist. Let us mention for instance the work ~\cite{MR2127744}, 
where Agrachev and Sarychev control Navier-Stokes equations by means of low modes.
They use methods of differential geometric~/~Lie algebraic control theory for finite 
dimensional control systems.

\skipline

The main difficulty of Open problem 1 is the behavior of the system near
$\partial \Omega \setminus \Gamma$. Indeed, although inertial forces prevail
inside the domain, viscous forces play a crucial role near the uncontrolled
boundary, and give rise to a boundary layer. An example of such a
phenomenon can be found in~\cite{MR1393067} where Coron derives an approximate
controllability result and highlights the creation of a boundary residue.
Hence, the key question is whether one can handle such a boundary layer
by means of the control.

Some authors have tried to study simplified geometries for
Open problem 1. In~\cite{MR2560050}, Chapouly studies a Navier-Stokes equation on 
a rectangle with Navier-slip boundary conditions on the uncontrolled part of the
boundary. She obtains small time global null controllability. In 
~\cite{MR2269867} and~\cite{MR2994698}, Guerrero, Imanuvilov and Puel 
prove approximate controllability for a  Navier-Stokes system in a square 
(resp. in a cube) where one side (resp. one face) is not controlled and has zero 
Dirichlet boundary condition.

\skipline

Burgers' equation has been extensively used as a toy model to investigate 
properties of more complex systems in a rather simple setting. This equation
was introduced in the seminal paper~\cite{MR0001147} by Burgers. Both from a 
theoretical and a numerical point of view, it already exhibits some key 
behaviors (such as interaction between the non-linearity and the smoothing 
effect). Therefore, our Theorem~\ref{thm.frederic} can be seen as an example 
for fast global null controllability despite the presence of a Dirichlet boundary 
layer. Moreover, despite the simplicity of Burgers' equation, the analogy 
between systems~(\ref{burgers2}) and~(\ref{navier.stokes}) is quite striking.
We can interpret our scalar control $u(t)$ as some one-dimensional counterpart
of a pressure gradient for 2D or 3D.

% ==============================================================================
\subsection{Previous works concerning Burgers' controllability}

Concerning the controllability of the inviscid Burgers' equation, some works
have be carried out. In~\cite{MR1616586}, Ancona and Marson describe the set of
attainable states in a pointwise way for the Burgers' equation on the half-line 
$x \geq 0$ with only one boundary control at $x = 0$. In~\cite{MR1612027}, 
Horsin describes the set of attainable states for a Burgers' equation on a line
segment with two boundary controls. Thorough studies are also carried out 
in~\cite{adimurthi} by Adimurthi et al. In~\cite{perrollaz}, Perrollaz studies the
controllability of the inviscid Burgers' equation in the context of entropy 
solutions with the additional control~$u(\cdot)$.

\skipline

Let us recall known results concerning the controllability of the
viscous Burgers' equation. We start with some positive results.

First, Fursikov and Imanuvilov have shown in~\cite{MR1406566} a small time 
local controllability result. It concerns local controllability in the vicinity
of trajectories of system~(\ref{burgers2}) and it only requires one boundary
control (either $y(t,0)$ or $y(t,1)$). Their proof relies on Carleman estimates 
for the parabolic problem obtained by seeing the non-linear term $yy_x$ as a 
small forcing term.

Global controllability towards steady states of system~(\ref{burgers2}) is
possible in large time both with one or two boundary controls. Such studies have
be carried out by Fursikov and Imanuvilov in~\cite{MR1348646} for large
time global controllability towards all steady states, and by Coron in
\cite{MR2376661} for global null-controllability in bounded time (ie. bounded 
with respect to the initial data).

When three scalar controls (namely $u(t)$, $y(t,0)$ and $y(t,1)$) are used,
Chapouly has shown in~\cite{MR2516179} that the system is small time 
exactly controllable to the trajectories. Her proof relies on the return method
and on the fact that the corresponding inviscid Burgers' system is small time 
exactly controllable (see~\cite[Chapter 6]{MR2302744} for other examples of this
method applied to Euler or Navier-Stokes).

\skipline

Some negative results have also been obtained.

In the context of only one boundary control $y(t,1)$, first obstructions where 
obtained by Diaz in~\cite{MR1364638}. He gives a restriction for the set of 
attainable states. Indeed, they must lie under some limit state corresponding
to an infinite boundary control $y(t,1) = + \infty$.

Still with only one boundary control, Fern\'andez-Cara and Guerrero 
derived an asymptotic of the minimal null-controllability time $T(r)$ for 
initial states of $H^1$ norm lower than $r$ (see~\cite{MR2311198}). This shows
that the system is not small-time controllable.

Guerrero and Imanuvilov have shown negative results in~\cite{MR2371111}
when two boundary controls $y(t,0)$ and $y(t,1)$ are used. They prove that
neither small time null controllability nor bounded time global 
controllability hold. Hence, controlling the whole boundary does not 
provide better controllability.

% ==============================================================================
\subsection{Strategy for steering the system towards the null state}

In view of these results, it seems that the pressure-like control $u(t)$
introduced by Chapouly is the key to obtaining small time global controllability 
results. In order to take advantage of both hyperbolic and parabolic behaviors 
of system~(\ref{burgers2}), our strategy consists in splitting the motion in 
three stages:

\skipline
\textbf{Hyperbolic stage}: Fast and approximate control towards the null state. 
During this very short stage $t \in [0, \varepsilon T]$ where $0 < \varepsilon \ll 1$, 
the systems behaves like the corresponding hyperbolic one, as the viscous term
does not have enough time to act. This hyperbolic system is small time null controllable.
During this first stage, we will use both $u(\cdot)$
and $v(\cdot)$ to try to get close to the null state, 
except for a boundary layer at $x = 1$.

\textbf{Passive stage}: Waiting. At the end of the first stage, we reach a state
whose size is hard to estimate due to the presence of a boundary layer. 
During this stage, we use null controls $v(t) = u(t) = 0$. Regularization 
properties of the viscous Burgers equation dissipate the boundary layer and the 
size of $y(t, \cdot)$ decreases. We show that it tends to zero in $L^2(0,1)$ when
$\varepsilon \rightarrow 0$. This is a crucial stage as is enables us to get rid of
the boundary residue. It seems to be a new idea and could also be applied for 
other boundary layers created when trying to get fast global controllability results.

\textbf{Parabolic stage}: Local exact controllability in the vicinity of zero.
After the two first stages, we succeed in getting very close to the null state.
The non-linear term becomes very small compared to the viscous one, and the
system now behaves like a parabolic one. We use a small time local exact 
controllability result to steer the system exactly to zero. 
During this last stage, we only need the control $v(\cdot)$.

\skipline

Most of the work to be done consists in deriving precise estimates for the creation
and the dissipation of the boundary layer. We will use the Cole-Hopf transform
(introduced in~\cite{MR0042889} and~\cite{MR0047234}) and Fourier series to overcome this difficulty. 
First, we will investigate the hyperbolic limit system (see Section \ref{section.limite}). 
Then we will derive estimates for the creation of the boundary layer during our
hyperbolic stage (see Section \ref{section.hyperbolique}) and estimates for its dissipation during
the passive stage (see Section \ref{section.passive}). 
This will achieve the proof of a small time global approximate null controllability result
for our system~(\ref{burgers2}). In Section \ref{section.parabolic}, we will
explain the parabolic stage and the local exact controllability.

% ==============================================================================
\subsection{A comparison lemma for controlled Burgers' systems}

Throughout our work, we will make an extensive use of the following comparison
lemma for our Burgers' system, in order to derive precise estimates. 
When the viscosity is null, this comparison principle still holds for entropy
solutions (as they are obtained as a limit of low viscosity solutions).

\begin{lemma}
\label{lemma.comparison}
Let $T, \nu > 0$ and consider $y_0, \hat{y}_0 \in L^2(0,1)$, 
$u, \hat{u} \in L^2(0,T)$, $v_0, \hat{v}_0, v_1, \hat{v}_1, \in H^{1/4}(0,T)$. 
Assume these data satisfy the following conditions:
\begin{equation*}
y_0 \leq \hat{y}_0 \quad \textrm{and} \quad
u \leq \hat{u} \quad \textrm{and} \quad
v_0 \leq \hat{v}_0 \quad \textrm{and} \quad
v_1 \leq \hat{v}_1.
\end{equation*}
Consider the following system (which is a generalized version of system 
(\ref{burgers2}):
\begin{equation}
	\label{burgers.3K}
	\left\{
	\begin{array}{rll}
		y_t + y y_x - \nu y_{xx} & = u(t)
		& \quad \mbox{in} \quad (0,T) \times (0,1), \\
		y(t, 0) & = v_0(t) & \quad \mbox{in} \quad (0,T), \\
		y(t, 1) & = v_1(t) & \quad \mbox{in} \quad (0,T), \\		
		y(0, x) & = y_0(x) & \quad \mbox{in} \quad (0,1).
	\end{array}
	\right.
\end{equation}
Then the associated solutions $y, \hat{y} \in X$ to system 
(\ref{burgers.3K}) are such that:
\begin{equation*}
y \leq \hat{y} \quad \textrm{on} \quad (0,T) \times (0,1).
\end{equation*}
\end{lemma}

One can find many comparison results in the litterature (see for instance the
book~\cite{MR2356201} and the references therein). However we give the proof
of Lemma \ref{lemma.comparison} both for the sake of completeness and because
with have not found this precise version anywhere.

\begin{proof}
We introduce $w = \hat{y} - y$. Thus, $w \in X$ is a solution to the system:
\begin{equation*}
	\left\{
	\begin{array}{rll}
		w_t - w_{xx} & = (\hat{u} - u) - \frac{1}{2}( w \hat{y} + w y)_x
		& \quad \mbox{in} \quad (0,T) \times (0,1), \\
		w(t, 0) & = \hat{v}_0(t) - v_0(t) & \quad \mbox{in} \quad (0,T), \\
		w(t, 1) & = \hat{v}_1(t) - v_1(t)  &	\quad \mbox{in} \quad (0,T), \\
		w(0, x) & = \hat{y}_0(x)- y_0(x) & \quad \mbox{in} \quad (0,1).
	\end{array}
	\right.
\end{equation*}
We want to study the negative part of $w$: $\delta = \min (w, 0)$. Hence, 
$\delta(t, 0) = \delta(t, 1) = 0$. Now we multiply the evolution equation by 
$\delta \leq 0$ and integrate by parts for $x \in [0,1]$ to get a $L^2$-energy 
estimate for $\delta$:
\begin{eqnarray*}
 \frac{1}{2} \frac{d}{dt} \int_0^1 \delta^2 + \nu \int_0^1 \delta_x^2
 & = & (\hat{u} - u)\int_0^1 \delta + \frac{1}{2} \int_0^1 \delta (\hat{y} + y)\delta_x  \\
 & \leq & \frac{\nu}{4} \int_0^1 \delta_x^2 + \frac{1}{4\nu} \int_0^1 \delta^2 (\hat{y} + y)^2 \\
 & \leq & \frac{\nu}{4} \int_0^1 \delta_x^2 + 
 \frac{1}{4\nu} \| \hat{y}(t,\cdot) + y(t, \cdot) \|^2_{\infty} \cdot \int_0^1 \delta^2.
\end{eqnarray*}
Thus, we can incorporate the first term of the right-hand side in the left-hand 
side:
\begin{equation*}
 \frac{1}{2} \frac{d}{dt} \int_0^1 \delta^2 
 \leq 
 \frac{1}{4\nu} \| \hat{y}(t,\cdot) + y(t, \cdot) \|^2_{\infty} \cdot \int_0^1 \delta^2.
\end{equation*}
Since $y, \hat{y} \in L^2\left((0,T);H^1(0,1)\right)$, we have that:
\begin{equation*}
 t \mapsto \| \hat{y}(t,\cdot) + y(t, \cdot) \|^2_{\infty}
 \quad \textrm{belongs to} \quad L^1(0,T).
\end{equation*}
Hence we can use Grönwall's lemma. Since $\delta(0,\cdot) \equiv 0$, we deduce 
that $\delta \equiv 0$ and $y \leq \hat{y}$.
\end{proof}

% ==============================================================================
\section{Analysis of the hyperbolic limit system}
\label{section.limite} 
% ==============================================================================

% ==============================================================================
\subsection{Small time versus small viscosity scaling}

Let us choose some $\varepsilon > 0$. We want to study what happens during the 
time interval $[0,\varepsilon T]$. To study this very short first stage, we 
perform the following change of scale. For $t \in [0,T]$ and $x \in [0,1]$, let:
\begin{equation}
\label{scaling}
\bar{y}(t,x) = \varepsilon y(\varepsilon t, x). 
\end{equation}
Hence, $\bar{y} \in X$ is now the solution to the small viscosity system:
\begin{equation}
	\label{burgers_epsilon}
	\left\{
	\begin{array}{rll}
		\bar{y}_t + \bar{y} \bar{y}_x - \varepsilon \bar{y}_{xx} & = \bar{u}(t)
		& \quad \mbox{in} \quad (0,T) \times (0,1), \\
		\bar{y}(t, 0) & = \bar{v}(t) &	\quad \mbox{in} \quad (0,T), \\
		\bar{y}(t, 1) & = 0 &	\quad \mbox{in} \quad (0,T), \\
		\bar{y}(0, x) & = \bar{y}_0(x) & \quad \mbox{in} \quad (0,1),
	\end{array}
	\right.
\end{equation}
where we performed the following scalings: 
$\bar{u}(t) = \varepsilon^2 u(\varepsilon t)$, 
$\bar{v}(t) = \varepsilon v(\varepsilon t)$
and $\bar{y}_0(x) = \varepsilon y_0(x)$. 
This scaling is fruitful because it highlights the fact that, when small
time scales are considered, the non-linear term is the key term.
We want to understand the behavior of the limit system when $\varepsilon = 0$.
Therefore, let us consider that $\bar{u}(\cdot)$, $\bar{v}(\cdot)$ and $\bar{y}_0(\cdot)$
are fixed data, and let $\varepsilon$ go to zero.

% ==============================================================================
\subsection{Obtaining the entropy limit}

When one considers the entropy limit $\varepsilon \rightarrow 0$ 
for system~(\ref{burgers_epsilon}), it is not possible to keep on enforcing strong
Dirichlet boundary conditions.  
A boundary layer appears and it is necessary to weaken the boundary conditions.
Otherwise, the system would become over-constrained.
The pioneer work concerning the derivation of such weak boundary conditions 
is the one by Bardos, Le Roux and Nédélec in~\cite{MR542510}.
In our particular setting, one gets the following system: 
\begin{equation}
	\label{b2_i}
	\left\{
	\begin{array}{rll}
		\bar{y}_t + \frac{1}{2} (\bar{y}^2)_x & = \bar{u}(t)
		& \quad \mbox{in} \quad (0,T) \times (0,1), \\
		\bar{y}(t, 0) & \in E(\bar{v}(t)) &	\quad \mbox{in} \quad (0,T), \\
		\bar{y}(t, 1) & \geq 0 & \quad \mbox{in} \quad (0,T), \\
		\bar{y}(0, x) & = \bar{y}_0(x) & \quad \mbox{in} \quad (0,1),
	\end{array}
	\right.
\end{equation}
where
\begin{equation*}
 E(\alpha) = \left\{ 
 \begin{array}{ll}
  ]-\infty; 0] & \textrm{if} \quad \alpha \leq 0, \\
   ]-\infty;-\alpha] \cup \{\alpha\} & \textrm{if} \quad \alpha > 0.
 \end{array}
 \right.
\end{equation*}

Let us explain the physical meaning of the set $E(\cdot)$. On the one hand, 
when one tries to enforce a negative boundary data on the left side, 
characteristics  instantly flow out of the domain, and our actions are useless. 
On the other hand, if we set a positive boundary data, then: either it is 
satisfied, or a greater negative wave overwhelms it.

Without getting into the details of entropy solutions (for that subject, refer 
to the definition given in~\cite{MR542510} or to the book~\cite{MR1707279}), we 
will use the following theorem  that guarantees that system~(\ref{b2_i}) is 
well-posed. 

\begin{theorem}[Bardos, Le Roux and Nédélec in~\cite{MR542510}]
\label{thm_bardos}
For any initial data $y_0 \in BV(0,1)$ and any pair of controls
$u \in L^1(0,T)$, $v \in BV(0,T)$, system~(\ref{b2_i}) has a
unique entropy solution $\bar{y}$ in the space $BV((0,1) \times (0,T))$. 
\end{theorem}

% ==============================================================================
\subsection{Small time null controllability}

We are going to show a small time null controllability result for
the hyperbolic limit system. 
However, this will not imply small time global
controllability since the system is not time reversible. Indeed, 
even though the PDE seems time-reversible, the definition of an entropy
solution is not.

\begin{theorem}
System~(\ref{b2_i}) is small time globally exactly null controllable.
\end{theorem}

Let us start by giving the intuition of the proof. In a first step, we enforce 
a constant left boundary data $H > 0$. It moves towards the right and overrides 
the initial data $\bar{y}_0(\cdot)$ provided that the shocks' propagation speed 
is sufficient. Therefore, $H$ is chosen by using the Rankine-Hugoniot formula. 
Figure~\ref{fig.godunov} shows a simulation of this first step for some smooth
initial data $\bar{y}_0$.
At the end of this step, we have $\bar{y}(\cdot) \equiv H$. During the second 
step, we use some constant negative $\bar{u}$ to get back down to the null 
state.

\begin{figure}[!ht]
\centering
\input{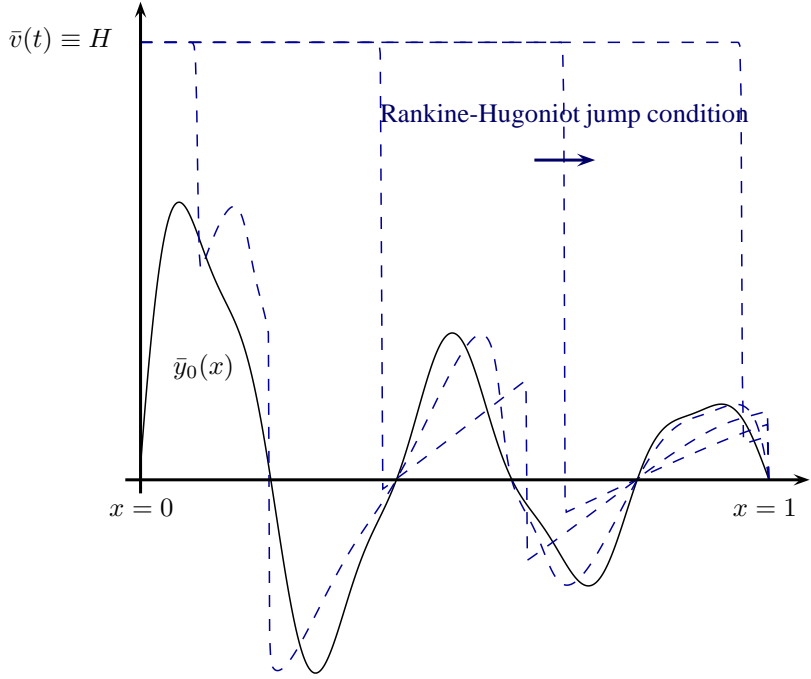}
\caption{Overriding of an initial data $\bar{y}_0(x)$ by some constant
state $\bar{y}(x) \equiv H$ for system~(\ref{b2_i}).}
\label{fig.godunov}
\end{figure}

Now let us give a rigorous proof using the comparison principle.

\begin{proof}
Let $\bar{y}_0(x) \in BV(0,1)$ and $T > 0$. Let us choose $H$ such that:
\begin{equation}
\label{condition_H}
 \frac{1}{2} \left( H - \|\bar{y}_0\|_{L^\infty} \right) \geq \frac{2}{T}.
\end{equation}
We enforce the following controls:
\begin{eqnarray*}
 \bar{v}(t) & = &  \left\{
 \begin{array}{ll}
  H & \textrm{for} \quad t \in [0,T/2], \\
  2H \left( 1 - \frac{t}{T} \right) & \textrm{for} \quad t \in [T/2, T],
 \end{array}
 \right. \\
 \bar{u}(t) & = & \left\{
 \begin{array}{ll}
  0 & \textrm{for} \quad t \in [0,T/2], \\
  - \frac{2H}{T} & \textrm{for} \quad t \in [T/2, T].
 \end{array}
\right. 
\end{eqnarray*}
From Theorem~\ref{thm_bardos}, we know that there exists a unique entropy solution
$\bar{y} \in BV((0,1)\times(0,T))$ for these data.
Let us show that $\bar{y}(T/2, \cdot) \equiv H$. 
Therefore, we will easily deduce $\bar{y}(T, \cdot) \equiv 0$.

\skipline

Let us extend our initial data from $[0,1]$ to $\R$. Since Theorem~\ref{thm_bardos}
guarantees the uniqueness of the solution, the restriction to $x \in [0,1]$ of our
global solution will be the unique solution to~(\ref{b2_i}).
Therefore we consider $\hat{y}_0 \in BV(\R)$:
\begin{equation}
 \hat{y}_0(x) = \left\{
 \begin{array}{ll}
   H & \textrm{for} \quad x < 0, \\
   \bar{y}_0(x) & \textrm{for} \quad 0 < x < 1, \\
   0 & \textrm{for} \quad 1 < x.   
 \end{array}
 \right.
\end{equation}
Let us introduce $\hat{y}$ the weak entropy solution defined on $\R \times [0,T]$
associated to this initial data. Thanks to Rankine-Hugoniot formula and
(\ref{condition_H}), we know that:
\begin{equation*}
 y(t,x) = H \quad \textrm{for} \quad x < t \cdot \frac{\left(H - \|y_0\|_{\infty}\right)}{2}.
\end{equation*}
Hence, $\hat{y}(T/2, x) = H$ for $x \in [0,1]$, and $y(t, 0^+) \equiv H$. If we 
want the restriction of $\hat{y}$ to be a solution
to~(\ref{b2_i}), we need to check that $y(t, 1^-) \geq 0$. Let us use the comparison
principle for solutions to inviscid Burgers' equation. It can be obtained by 
taking the null viscosity limit in our Lemma~\ref{lemma.comparison}.
Hence $\hat{y}(t,x) \geq w(t,x)$ where $w$ is the solution associated to the initial
data:
\begin{equation}
 w_0(x) = \left\{
 \begin{array}{ll}
   H & \textrm{for} \quad x < 0, \\
   - \|\bar{y}_0\|_{\infty} & \textrm{for} \quad 0 < x < 1, \\
   0 & \textrm{for} \quad 1 < x.   
 \end{array}
 \right.
\end{equation}
We have two Riemann problems. Near $x = 1$, we have a rarefaction wave. 
Hence $x \mapsto w(t,x)$
is continuous near $x = 1$ as long as the $H$ shock wave has not reached $x=1$.
Hence $w(t, 1^-) = 0$ before $T^* = 1/(2H - 2\|\bar{y}_0\|_{\infty})$, then
$w(t, 1^-) = H$. This is why $w(t, 1^-) \geq 0$. Thus
$\hat{y}(t, 1^-) \geq w(t, 1^-) \geq 0$. The restriction $\hat{y}_{[0,1]}$
is the unique solution to~(\ref{b2_i}) and it is equal to $H$ at time $t = T/2$.
\end{proof}

This proof uses the comparison principle for Burgers' equation. Since we consider
a 1-D system, this is not a problem. However, if we wanted to be able to handle
multi-dimensional systems, we could use the generalized characteristics method
from Dafermos (see~\cite{MR0457947}). This technique has been successfully used
by Perrollaz in~\cite{perrollaz}.

% ==============================================================================
\section{Hyperbolic stage and settling of the boundary layer}
\label{section.hyperbolique}
% ==============================================================================

Thanks to the analysis of the hyperbolic limit system, we were able to exhibit
controls steering the system towards the null state from any initial data. Now 
we want to apply the same strategy to the slightly viscous system 
(\ref{burgers_epsilon}) by using very similar controls. However, a boundary 
layer is going to appear. Our goal in this section is to derive bounds for the 
boundary layer at the end of this stage.

% ==============================================================================
\subsection{Steady states of system~(\ref{burgers_epsilon})}

From now on, the viscosity is positive. Hence, since we have a zero Dirichlet 
boundary condition $\bar{y}(1) = 0$, we cannot hope to reach a constant state 
$\bar{y}(x) \equiv H > 0$ . However, we expect that we can get very close to 
the  corresponding steady state. Let us introduce the following steady state 
of system~(\ref{burgers_epsilon}):
\begin{equation}
  \label{def.h}
  h^{\varepsilon}(x) = H \tanh\left(\frac{H}{2\varepsilon}(1-x)\right).
\end{equation}

\begin{lemma}
\label{lemma.steady}
For any $H > 0$ and any $\varepsilon > 0$, $h^{\varepsilon}$ defined by 
(\ref{def.h}) is a stationary solution to system~(\ref{burgers_epsilon}) with 
controls: $\bar{u}(t) = 0$ and 
$\bar{v}(t) = H \tanh\left(\frac{H}{2\varepsilon}\right)$.
\end{lemma}

\begin{proof}
The proof is an easy computation. In fact, it is possible to compute explicitly
all the steady states for system~(\ref{burgers_epsilon}), at least when 
$\bar{u} = 0$. This is done in~\cite{MR1348646} with viscosity 
$\varepsilon = 1$. 
\end{proof}

We have chosen a boundary data 
$\bar{v}(t) = H \tanh\left(\frac{H}{2\varepsilon}\right)$ for the definition of 
our steady state $h^{\varepsilon}$, but we will use a control $\bar{v}(t) = H$ 
for the motion. This technical trick will lighten some computations and is 
relevant since both terms are exponentially close as $\varepsilon$ goes to zero. 
However, some proofs require the use of the exact steady state corresponding
to a boundary data $\bar{v}(t) = H$. Therefore, we introduce:
\begin{equation}
  \label{def.k}
  k^{\varepsilon}(x) = 
  K \tanh\left(\frac{K}{2\varepsilon}(1-x)\right),
\end{equation}
where $K > 0$ is given by the implicit relation 
$K \tanh \left( K/(2\varepsilon) \right) = H$.

\begin{lemma}
For any $H > 0$ and any $\varepsilon > 0$, $k^{\varepsilon}$ defined by 
(\ref{def.k}) is a stationary solution to system~(\ref{burgers_epsilon}) with 
controls: $\bar{u}(t) = 0$ and $\bar{v}(t) = H$. Moreover, we have the 
estimate:
\begin{equation}
\label{estimate.steady}
\left\| k^{\varepsilon} - h^{\varepsilon} \right\|_{L^\infty(0,1)}
\leq 2 H e^{-H/\varepsilon}.
\end{equation}
\end{lemma}

\begin{proof}
Lemma \ref{lemma.steady} gives us that $k^{\varepsilon}$ is a steady state. 
For the estimate, we write:
\begin{eqnarray*}
\left\| k^{\varepsilon} - h^{\varepsilon} \right\|_{L^\infty(0,1)}
& \leq & \left| K \tanh\left(\frac{K}{2\varepsilon}\right)
- H \tanh\left(\frac{H}{2\varepsilon}\right) \right| \\
& \leq & H \left| 1 - \tanh\left(\frac{H}{2\varepsilon}\right) \right| \\
& \leq & 2 H e^{-H/\varepsilon}.
\end{eqnarray*}
\end{proof}

% ==============================================================================
\subsection{First step: overriding the initial data}

In order to get close to the steady state $h^{\varepsilon}$, it is necessary to 
choose $H$ in such a way that a Rankine-Hugoniot type condition is satisfied. 
Once we get close enough to the steady state, the solution will very quickly 
converge to the steady state. Indeed, the eigenvalues of the linearized system 
around this steady state are real, negative, and of size at least 
$1/\varepsilon$. This guarantees very quick convergence to the steady state. 
Such a study of the linearized problem around a steady state for the Burgers' 
equation can be found in~\cite{MR863984}. We give the following lemma 
describing the settling of the limit layer.

\begin{lemma}
\label{lemma.settling}
Let $T > 0$, $H > 0$ and $y_0 \in H^1_0(0,1)$ be given data. 
Then for $\varepsilon > 0$ small enough, there exists a boundary 
control  $\bar{v} \in H^{3/4}(0,T)$  such that $\bar{v}(\cdot) \leq H$ and such
that the solution $\bar{y} \in X$ to system~(\ref{burgers_epsilon}) 
with initial data $\bar{y}_0 = \varepsilon y_0$ and controls $\bar{u} = 0$ and 
$\bar{v}$ satisfies:
\begin{equation}
\label{estimate.lemma.settling}
\left\| \bar{y}(T, \cdot) - h_{\varepsilon}(\cdot) \right\|_{L^2(0,1)}
= \mathcal{O}_{\varepsilon \rightarrow 0} 
\left( \varepsilon^{-1/2} e^{-\frac{H}{4\varepsilon}\left(HT - 2\right)} \right).
\end{equation}
\end{lemma}

\begin{figure}[!ht]
\centering
\input{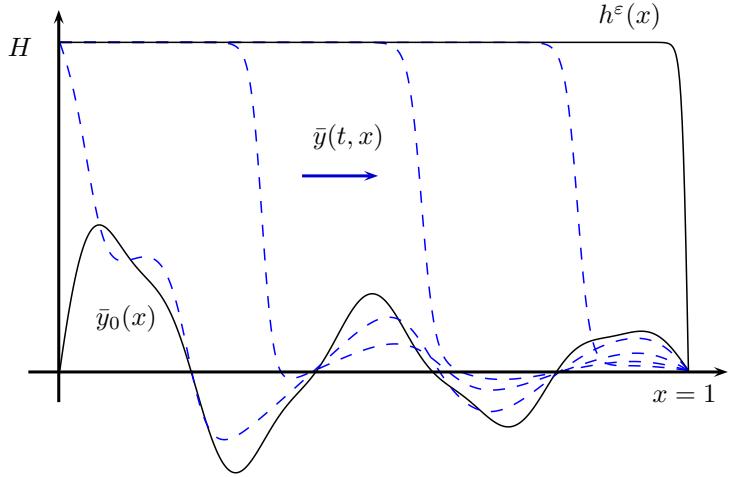}
\caption{Example of evolution from an initial data towards a steady state.}
\end{figure}

Let us postpone the proof of Lemma~\ref{lemma.settling} for the moment. We start
by giving a few remarks concerning this statement and its proof.
The intuition is to choose a boundary control $\bar{v}(t) \equiv H$, just like
we have done for the hyperbolic case. Moreover, we want to use the Cole-Hopf 
transform and Fourier series to compute explicitly $\bar{y}(T, \cdot)$. Let us 
introduce the Cole-Hopf transform:
\begin{equation*}
Z(t,x) = \exp\left( - \frac{1}{2\varepsilon} \int_0^x \bar{y}(t,s) ds \right).
\end{equation*}
This leads to the following heat system for the new unknown $Z$: 
\begin{equation}
\label{system.HC1}
\left\{
	\begin{array}{rll}
		Z_t - \varepsilon Z_{xx} & = 
		- \left(\frac{1}{4\varepsilon} \bar{y}^2(t,0) 
		- \frac{1}{2}\bar{y}_x(t,0) \right) Z
		& \quad \mbox{on} \quad (0,T) \times (0,1), \\
		Z(t, 0) & = 1 &	\quad \mbox{on} \quad (0,T), \\
		Z_x(t, 1) & = 0 & \quad \mbox{on} \quad (0,T), \\
		Z(0, x) & = Z^0(x) & \quad \mbox{on} \quad (0,1),
	\end{array}
	\right.
\end{equation}
where the initial data $Z^0$ is computed from the initial data 
$\bar{y}_0 = \varepsilon y_0$: 
\begin{equation}
\label{def.Z0}
 Z^0(x) = \exp \left( - \frac{1}{2} \int_0^x y_0(s) ds \right).
\end{equation}
Hence we see that it will not be possible to carry on explicit computations if
we do choose $\bar{y}(t,0) \equiv H$. Indeed, in that case, we would not know
explicitly $\bar{y}_x(t,0)$ (which is needed to compute the solution to system
(\ref{system.HC1})). However, we are confident that this term is very small.
Hence, we are going to go the other way around: we will choose our control 
explicitly in the Cole-Hopf domain and use it to compute our control 
$\bar{v}(\cdot)$. Therefore, we are interested in the following heat system:
\begin{equation}
\label{system.HC2}
\left\{
	\begin{array}{rll}
		Z_t - \varepsilon Z_{xx} & = - \frac{H^2}{4\varepsilon} Z
		& \quad \mbox{on} \quad (0,T) \times (0,1), \\
		Z(t, 0) & = 1 &	\quad \mbox{on} \quad (0,T), \\
		Z_x(t, 1) & = 0 & \quad \mbox{on} \quad (0,T), \\
		Z(0, x) & = Z^0(x) & \quad \mbox{on} \quad (0,1).
	\end{array}
	\right.
\end{equation}
If we go back to the Burgers' domain, this means that we somehow use the 
following boundary condition at $x = 0$: 
\begin{equation}
\label{relation.HC}
\bar{y}_x(t,0) = \frac{1}{2 \varepsilon} \left( \bar{y}^2(t,0) - H^2\right).
\end{equation}
We expect that the solution $Z$ will converge towards $H^{\varepsilon}(\cdot)$,
where $H^{\varepsilon}(\cdot)$ is the Cole-Hopf transform of the steady state
$h^{\varepsilon}$:
\begin{equation}
H^{\varepsilon}(x) = \frac{\cosh \left( \frac{H}{2\varepsilon}(1-x) \right)}
{\cosh \frac{H}{2\varepsilon}}.
\end{equation}
Indeed, we have the following lemma.

\newcommand{\Xdeux}{L^2((0,T); H^3(0,1)) \cap H^1((0,T);H^1(0,1))}
\begin{lemma}
\label{lemma.HC}
Let $T > 0$ and $Z^0 \in H^2(0,1)$ such that $Z^0(0) = 1$ and $Z^0_x(1) = 0$.
Then system~(\ref{system.HC2}) has a unique solution $Z$ in the space
$L^2((0,T); H^3(0,1)) \cap H^1((0,T);H^1(0,1))$. Moreover, there exists a 
constant $C(Z^0) > 0$ depending only on $\|Z^0\|_{H^1}$ such that:
\begin{equation}
\label{estimate.lemma.HC}
\left\| Z(T, \cdot) - H^{\varepsilon}(\cdot) \right\|_{H^1(0,1)} \leq
\varepsilon^{-1/2} C(Z^0) e^{-\frac{H^2T}{4\varepsilon}}.
\end{equation}
\end{lemma}

\begin{proof}
It is classical to show that system~(\ref{system.HC2}) has a unique solution in
the space $L^2((0,T); H^3(0,1)) \cap H^1((0,T);H^1(0,1))$. One can even get
more smoothness if needed. An efficient method is the semi-group method that
one can find for instance in~\cite{MR710486}.
To compute the dynamics of system~(\ref{system.HC2}), we introduce the adequate
Fourier basis of $L^2$:
\begin{equation*}
f_n(x) = \sqrt{2} \sin \left(\left( n + \frac{1}{2} \right) \pi x\right)
\quad \textrm{for} \quad n \geq 0.
\end{equation*}
Hence $f_n(0) = f_n'(1) = 0$. We will use the notation 
$\lambda_n = (n + \frac{1}{2}) \pi$. Thus, $f_n'' = - \lambda_n^2 f_n$.
Let us give the following scalar products, which can easily be computed using
integration by parts:
\begin{eqnarray}
\langle 1 | f_n \rangle & = & \frac{\sqrt{2}}{\lambda_n}, \nonumber \\
\langle H^{\varepsilon} | f_n \rangle & = & 
  \frac{\sqrt{2}\lambda_n}{\frac{H^2}{4\varepsilon^2} + \lambda_n^2}, \label{scalar1} \\
\left| \langle Z^0 | f_n \rangle \right| & \leq & 
  \frac{\sqrt{2}}{\lambda_n} \left( 1 + \frac{1}{2} 
  \left\| Z^0 \right\|_{H^1} \right). \label{scalar2}
\end{eqnarray}
In these equations $\langle \cdot | \cdot  \rangle$ denotes the standard 
scalar product in $L^2(0,1)$. 
Let us write $Z = 1 + w$. Hence $w$ will satisfy $w(t,0) = w_x(t,1) = 0$.
Easy computations lead to the following ordinary differential equations
for the components of $w$ on our Fourier basis:
\begin{equation*}
 \dot{w}_n(t) = - \varepsilon\left(\lambda_n^2 + \frac{H^2}{4\varepsilon^2}\right)w_n(t)
 - \frac{H^2}{4 \varepsilon} \langle 1 | f_n \rangle.
\end{equation*}
It is easy to see that the fixed points for these ODEs are the expected 
coefficients $\langle H^{\varepsilon} - 1 | f_n \rangle$. We can solve these 
ODEs with our initial condition:
\begin{equation*}
 w_n(t) = \alpha_n
 e^{-\varepsilon\left(\lambda_n^2 + \frac{H^2}{4\varepsilon^2}\right)t}
 + \langle H^{\varepsilon} - 1 | f_n \rangle,
\end{equation*}
where:
\begin{equation*}
 \alpha_n = \langle Z^0 | f_n \rangle - \langle H^{\varepsilon} | f_n \rangle.
\end{equation*}
Now we can estimate $Z(T, \cdot) - H^{\varepsilon}(\cdot)$:
\begin{equation*}
\left\| Z(T, \cdot) - H^{\varepsilon}(\cdot) \right\|_{H^1(0,1)}^2
= \sum_{n \geq 0} \lambda_n^2 \alpha_n^2 
e^{-2 \varepsilon\left(\lambda_n^2 + \frac{H^2}{4\varepsilon^2}\right)T}.
\end{equation*}
From the expression of $\alpha_n$, (\ref{scalar1}) and (\ref{scalar2}) we get 
the easy bound:
\begin{equation*}
 \lambda_n^2 \alpha_n^2 \leq 16 + \|Z^0\|^2_{H^1(0,1)}, 
 \quad \forall n \in \N.
\end{equation*}
Thus, we get
\begin{equation*}
\left\| Z(T, \cdot) - H^{\varepsilon}(\cdot) \right\|_{H^1(0,1)}^2
\leq \left(16 + \|Z^0\|^2_{H^1(0,1)}\right) e^{-\frac{H^2T}{2\varepsilon}} 
\sum_{n \geq 0} e^{-2 \varepsilon \lambda_n^2}.
\end{equation*}
Now we split the sum in two parts: 
$n\leq N = \lfloor1/\varepsilon\rfloor$ and $n \geq N$. We get:
\begin{eqnarray*}
\left\| Z(T, \cdot) - H^{\varepsilon}(\cdot) \right\|_{H^1(0,1)}^2
& \leq & \left(16 + \|Z^0\|^2_{H^1(0,1)}\right) e^{-\frac{H^2T}{2\varepsilon}} \left( N + 
\sum_{k \geq 0} e^{-2 \varepsilon (N + k + \frac{1}{2})^2 \pi^2} \right) \\
& \leq & \left(16 + \|Z^0\|^2_{H^1(0,1)}\right) e^{-\frac{H^2T}{2\varepsilon}} \left( N + 
\frac{1}{1-e^{-4 \varepsilon N \pi^2}} e^{-2\varepsilon N^2 \pi^2} \right).
\end{eqnarray*}
Hence, for $\varepsilon$ small enough, we have:
\begin{equation*}
\left\| Z(T, \cdot) - H^{\varepsilon}(\cdot) \right\|_{H^1(0,1)}^2 \leq
\left(\frac{1}{\varepsilon} + 1\right) 
\left(16 + \|Z^0\|^2_{H^1(0,1)}\right) e^{-\frac{H^2T}{2\varepsilon}}.
\end{equation*} 
This concludes the proof of Lemma~\ref{lemma.HC}.
\end{proof}

Now we can prove Lemma~\ref{lemma.settling}.

\begin{proof}[Proof of Lemma~\ref{lemma.settling}]
\textbf{Definition of the control:} \quad
Using Lemma~\ref{lemma.HC}, we start by considering the solution $Z \in \Xdeux$ 
to system~(\ref{system.HC2}) with the initial data~(\ref{def.Z0}). Since 
$Z_0(\cdot) > 0$, the usual strong maximum principle (see~\cite{MR1707279}) 
guarantees that $Z(t,x) > 0$. Thus, we can define:
\begin{equation}
\label{def.bary}
\bar{y}(t,x) = - 2 \varepsilon \frac{Z_x(t,x)}{Z(t,x)}.
\end{equation}
Hence $\bar{y} \in X$ is a solution to~(\ref{burgers_epsilon}) with initial data 
$\varepsilon y_0$ and boundary control $\bar{v}(t) = - 2 \varepsilon Z_x(t, 0)$. 
Since $Z \in \Xdeux$, we can show that its boundary trace $Z_x(t, 0)$ belongs to 
$H^{3/4}(0,T)$. Hence $\bar{v} \in H^{3/4}(0,T)$.

\skipline

\textbf{Proof of an $L^\infty$ bound on the solution:} \quad
If $\varepsilon$ is small enough, then $\varepsilon \|y_0\|_{\infty} \leq H$. 
Moreover, we know that $\bar{v} \in H^{3/4}(0,T)$. Hence, 
$\bar{v} \in \mathcal{C}^0[0,T]$. Assume that $\sup_{[0,T]} \bar{v} > H$. Let
$T_0$ be a time such that $\bar{v}(T_0) = \sup_{[0,T]} \bar{v} > H$. On the
one hand, by the comparison principle from Lemma~\ref{lemma.comparison}, we 
know that:
\begin{equation}
\label{ineq.HC}
\bar{y} \leq \bar{v}(T_0) \quad \text{on} \quad (0,T) \times (0,1).
\end{equation}
On the other hand, we recall relation~(\ref{relation.HC}):
\begin{equation*}
\bar{y}_x(t,0) = \frac{1}{2 \varepsilon} \left( \bar{y}^2(t,0) - H^2\right).
\end{equation*}
Hence, since $\bar{v}(T_0) > 0$, we get
$\bar{y}_x(T_0,0) > 0$. Thus,there exists $x > 0$ such that 
$\bar{y}(T_0,x) > \bar{v}(T_0) = \sup_{[0,T]} \bar{v}$.
This is in contradiction with assertion~(\ref{ineq.HC}).
Hence, if $\varepsilon$ is small enough, 
$\bar{v}(\cdot) \leq H$ and $\bar{y}(T, \cdot) \leq H$.

\skipline

\textbf{Derivation of the $L^2$ estimate at time $T$:} \quad
Now we want to prove estimate~(\ref{estimate.lemma.settling}) from Lemma~
\ref{lemma.settling}. We want to use estimate~(\ref{estimate.lemma.HC}) from
Lemma~\ref{lemma.HC}. We perform the following computation at time $T$ and for 
any $x \in [0,1]$:
\begin{eqnarray*}
\left| \bar{y} - h^{\varepsilon} \right| 
=  2 \varepsilon \left| \frac{Z_x}{Z} - 
\frac{H^{\varepsilon}_x}{H^{\varepsilon}} \right| 
& = & 2\varepsilon \left| \frac{ Z \left( Z_x - H^{\varepsilon}_x \right) +
  Z_x \left( H^{\varepsilon} - Z \right) }
  { Z H^{\varepsilon} } \right| \\
& \leq & 2 \varepsilon \left| \frac{Z_x - H^{\varepsilon}_x}{H^{\varepsilon}} \right|
+ 2 \varepsilon \left| \frac{Z_x}{Z} \right| \cdot \left| 
\frac{Z - H^{\varepsilon}}{H^{\varepsilon}} \right|.
\end{eqnarray*}
Thus, we get:
\begin{equation*}
\left\| \bar{y}(T, \cdot) - h^{\varepsilon}(\cdot) \right\|_{L^2(0,1)}
\leq \left (2 \varepsilon + \|\bar{y}(T, \cdot)\|_{\infty}\right) \times
\sup_{[0,1]}\frac{1}{H^{\varepsilon}} \times
\left\| Z(T, \cdot) - H^{\varepsilon}(\cdot)\right\|_{H^1(0,1)}.
\end{equation*}
Now we use that $\| \bar{y}(T, \cdot) \|_{\infty} \leq H$ and
$\sup_{[0,1]} 1/H^{\varepsilon} \leq e^{+H/2\varepsilon}$. Hence
, using also (\ref{estimate.lemma.HC}),
\begin{equation*}
\left\| \bar{y}(T, \cdot) - h^{\varepsilon}(\cdot) \right\|_{L^2(0,1)} \leq
\frac{1}{\sqrt{\varepsilon}} \left(2 \varepsilon + H \right) 
C(Z^0) e^{-\frac{H}{4 \varepsilon}(HT - 2)}.
\end{equation*}
This estimate concludes the proof of Lemma~\ref{lemma.settling}.
\end{proof}

\begin{remark}
In Lemma~\ref{lemma.settling}, we take an initial data $y_0 \in H^1_0(0,1)$. 
This is a technical assumption that enables us to use stronger solutions.
We will get rid of it later on, by letting the Burgers' equation smooth our 
real initial data which is only in $L^2(0,1)$.
\end{remark}

% ==============================================================================
\subsection{Second step: going back to the null state}

Once we have reached the steady state $h^\varepsilon$, 
we wish to go back to the null state.
This is done by applying a suitable negative interior control $\bar{u}$. 
The control $\bar{v}$ will only be following the global movement. The intuitive
idea is to apply some negative control $\bar{u}$ on $[0,T]$ such that 
$\int_0^T u(t) dt = - H$. Thus, we hope to reach some state that is below $0$
and above a boundary residue $h^{\varepsilon} - H$. However, this last statement
is only true up to some small $L^2$ function (small as $T \rightarrow 0$).
The key will be to choose the duration $T$ of this step small enough 
(with respect to $\varepsilon$).

\begin{figure}[!ht]
\centering
\input{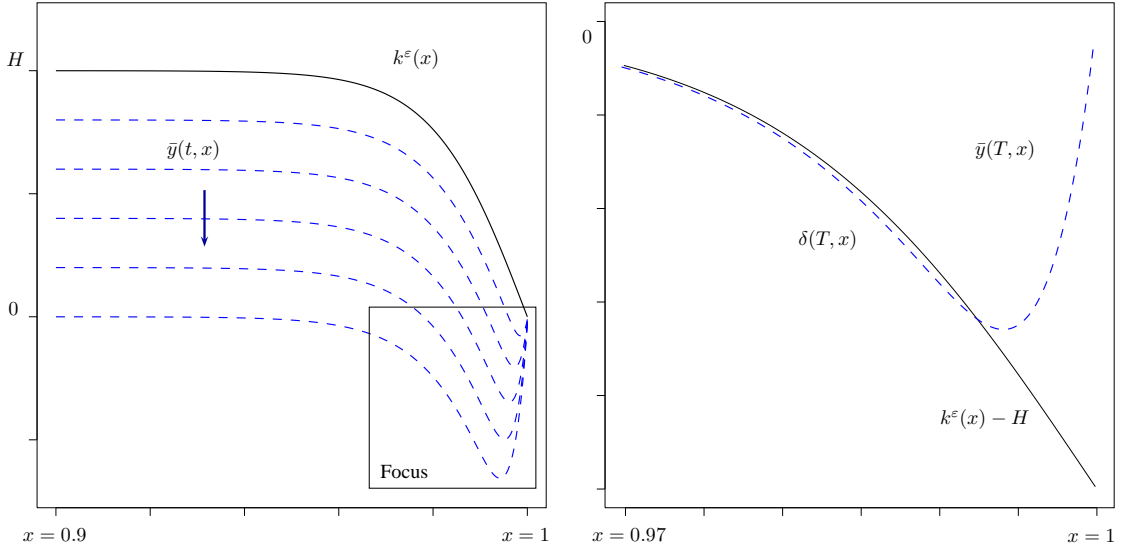}
\caption{Numerical simulation of the push-down towards
the null state and the creation of a boundary residue.
The final state $\bar{y}(T, \cdot)$ is almost above the
residue $k^{\varepsilon}(\cdot) - H$.}
\end{figure}

\begin{lemma}
\label{lemma.push}
Let $\varepsilon > 0$ and $H > 0$ be given data. 
Assume that $2 \varepsilon \leq H$. We consider the evolution of
an initial data $\bar{y}_1 \in L^2(0,1)$.
For any $T > 0$, we consider the following controls for $t\in [0,T]$:
\begin{eqnarray}
\label{def.u} \bar{u}(t) & = &  - \frac{H}{T}, \\
\label{def.v} \bar{v}(t) & = & H + \int_{0}^{t} u(s) ds.
\end{eqnarray}
Then the associated solution $\bar{y} \in X$ to system~(\ref{burgers_epsilon})
satisfies:
\begin{equation}
 \bar{y}(T, \cdot) - k^{\varepsilon}(\cdot) + H 
 \geq \delta(T, \cdot),
\end{equation}
where $\delta \in X$ is the solution to some Burgers-like system given below
and is such that:
\begin{equation}
\left\| \delta(T, \cdot) \right\|_{L^2} \leq 
e^{H^2T/4\varepsilon} \left\| \bar{y}_1 - k^{\varepsilon} \right\|_{L^2} 
+ 2 H \left(e^{H^2T/2\varepsilon} - 1 \right).
\end{equation}
\end{lemma}

\begin{proof} 
Let $T >0$ and consider the controls defined by~(\ref{def.u}) and~(\ref{def.v}).
Let us consider the associated solution $\bar{y}\in X$ to 
(\ref{burgers_epsilon}). 
We compare $\bar{y}$ to the solution $z \in X$ to the following system:
\begin{equation}
	\label{burgers.z}
	\left\{
	\begin{array}{rll}
		z_t + z z_x - \varepsilon z_{xx} & = \bar{u}(t)
		& \quad \mbox{in} \quad (0,T) \times (0,1), \\
		z(t, 0) & = \bar{v}(t) & \quad \mbox{in} \quad (0,T), \\
		z(t, 1) & = \bar{v}(t) - H & \quad \mbox{in} \quad (0,T), \\		
		z(0, x) & = \bar{y}_0(x) & \quad \mbox{in} \quad (0,1).
	\end{array}
	\right.
\end{equation}
The comparison principle from Lemma~\ref{lemma.comparison} tells us that 
$y(T, \cdot) \geq z(T, \cdot)$. Now we want to derive precise estimates for the
solution $z \in X$. We write:
\begin{equation}
z(t,x) = k^{\varepsilon}(x) + \int_0^t \bar{u}(s)ds + \delta(t, x),
\end{equation}
where $\delta \in X$ is thus the solution to the following system:
\begin{equation}
	\label{burgers.delta}
	\left\{
	\begin{array}{rll}
		\delta_t - \varepsilon \delta_{xx} 
		+ k^{\varepsilon} \delta_x + \left(\delta + \int_0^t \bar{u}(s)ds\right)
		 (k^{\varepsilon} + \delta)_x
		& = 0
		& \quad \mbox{in} \quad (0,T) \times (0,1), \\
		\delta(t, 0) & = 0 &	\quad \mbox{in} \quad (0,T), \\
		\delta(t, 1) & = 0 &	\quad \mbox{in} \quad (0,T), \\
		\delta(0, x) & = \bar{y}_1(x) - k^{\varepsilon}(x) & 
			\quad \mbox{in} \quad (0,1).
	\end{array}
	\right.
\end{equation}
Note that it is convenient in this proof to use $k^\varepsilon$ in order
to get exact zero boundary conditions $\delta(t,0) = \delta(t,1) = 0$.
We multiply the evolution equation of~(\ref{burgers.delta}) by $\delta$ and
integrate by parts for $x \in [0,1]$ to get a $L^2$-energy estimate on 
$\delta$:
\begin{eqnarray*}
\frac{1}{2}\frac{d}{dt}\int_0^1 \delta^2 + \varepsilon \int_0^1 \delta_x^2
& = & - \int_0^1 k^{\varepsilon} \delta \delta_x - 
\int_0^1 \left(\delta + \int_0^t \bar{u}(s)ds\right) 
(k^{\varepsilon} + \delta)_x \delta \\
& = & \frac{1}{2} \int_0^1 \delta^2 (k^{\varepsilon})_x -
\int_0^1 \left(\delta + \int_0^t \bar{u}(s)ds\right) \delta (k^{\varepsilon})_x  \\
& = & - \frac{1}{2} \int_0^1 \delta^2 (k^{\varepsilon})_x
- \int_0^t \bar{u}(s)ds \int_0^1 \delta (k^{\varepsilon})_x.
\end{eqnarray*}
Now we use definition~(\ref{def.k}) and the assumption $2\varepsilon \leq H$:
\begin{equation*}
\left\|k^{\varepsilon}_x\right\|_{\infty} 
\leq \frac{K^2}{2\varepsilon}
\leq \frac{H^2}{2\varepsilon \tanh(1)^2} 
\leq \frac{H^2}{\varepsilon}.
\end{equation*}
Moreover, $\int_0^t \bar{u}(s)ds \leq H$. Hence,
\begin{equation}
\frac{1}{2} \frac{d}{dt}\int_0^1 \delta^2 \leq
\frac{H^2}{2\varepsilon} \int_0^1 \delta^2
+ \frac{H^3}{\varepsilon} \left( \int_0^1 \delta^2 \right)^{1/2}.
\end{equation}
Let us denote $E(t) = \| \delta(t, \cdot) \|_{L^2}$. Hence, one has:
\begin{equation}
\dot{E}(t) \leq \frac{H^2}{2\varepsilon} E + \frac{H^3}{\varepsilon}.
\end{equation}
From Grönwall's lemma, we get:
\begin{equation}
E(T) \leq \left( E(0) + 2 H \right) e^{H^2T/2\varepsilon} - 2 H.
\end{equation}
This concludes the proof of Lemma~\ref{lemma.push}.
\end{proof}

This is the end of the hyperbolic stage. We need to perform the reverse scaling 
of~(\ref{scaling}) to go back to $y$ (and not $\bar{y}$). We have shown that
we are above some boundary residue $h^{\varepsilon} - H$. Hence, we have to study
the evolution of the following initial data:
\begin{equation}
\label{def.Phi}
 \Phi^{\varepsilon}(x) = \frac{1}{\varepsilon} \left(h^{\varepsilon}(x) - H\right) = 
 \frac{H}{\varepsilon} \left(\tanh\left(\frac{H}{2\varepsilon}(1-x)\right) - 1 \right).
\end{equation}

One should be scared by the size of this boundary residue that we are left with.
Indeed, its $L^2$ size grows like $1/\sqrt{\varepsilon}$. However it has the 
important feature that its typical wavelength is $\varepsilon$. Hence, its 
spectral decomposition will mostly involve high frequencies that will decay 
rapidly  during the passive stage thanks to smoothing effects of Burgers' 
equation.

% ==============================================================================
\section{Passive stage and dissipation of the boundary layer}
\label{section.passive}
% ==============================================================================

The goal of this section is to prove the following estimate concerning the 
dissipation of the boundary residue $\Phi^{\varepsilon}$ created in the previous 
section. Indeed, although its $L^2$-norm increases as $\varepsilon$ goes to 
zero, regularization effects of the Burgers equation will dissipate it in any 
positive time $T$.

\begin{lemma}
\label{lemma.dissolution}
Let $T > 0$ be a fixed positive time. For any $\varepsilon > 0$, let us 
consider $\phi \in X$ the solution to the following system:
\begin{equation*}
	\left\{
	\begin{array}{rll}
		\phi_t + \phi \phi_x - \phi_{xx} & = 0
		& \quad \mbox{in} \quad (0,T) \times (0,1), \\
		\phi(t, 0) & = 0 &	\quad \mbox{in} \quad (0,T), \\
		\phi(t, 1) & = 0 &	\quad \mbox{in} \quad (0,T), \\
		\phi(0, x) & = \Phi^{\varepsilon}(x) & \quad \mbox{in} \quad (0,1),
	\end{array}
	\right.
\end{equation*}
where $\Phi^{\varepsilon}(x)$ is the boundary residue defined by~(\ref{def.Phi}). 
Then for any $\delta > 0$, we have the estimate:
\begin{equation}
\label{majo.findissolution}
\| \phi(T, \cdot) \|_{L^2(0,1)} = \mathcal{O}_{\varepsilon \rightarrow 0}
\left(\varepsilon^{1-\delta}\right).
\end{equation}
\end{lemma}

\begin{figure}[!ht]
\centering
\input{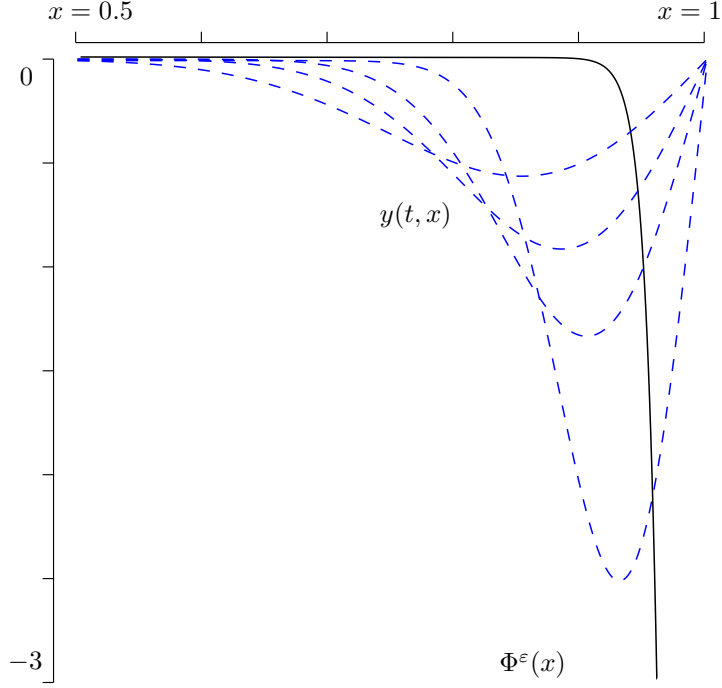}
\caption{Numerical simulation of the dissipation of the boundary residue
$\Phi^\varepsilon(\cdot)$. At time $t = 0$, the boundary residue was of size
$\|\Phi^\varepsilon(\cdot)\|_{\infty} = 100$.}
\end{figure}

% ==============================================================================
\subsection{Cole-Hopf transform}

Once again, we are going to use the Cole-Hopf transform to derive precise 
estimates. Therefore, let us introduce the following change of unknown for 
$x \in [0,1]$ and $t \in [0,T]$:
\begin{equation*}
 z(t,x) = \exp\left( - \frac{1}{2} \int_0^x \phi(t,s) ds \right).
\end{equation*}
This leads to the following heat system for the new unknown $z$: 
\begin{equation}
\label{system.hc}
\left\{
	\begin{array}{rll}
		z_t - z_{xx} & = 0
		& \quad \mbox{on} \quad (0,T) \times (0,1), \\
		z_x(t, 0) & = 0 &	\quad \mbox{on} \quad (0,T), \\
		z_x(t, 1) & = 0 & \quad \mbox{on} \quad (0,T), \\
		z(0, x) & = Z^{\varepsilon}(x) & \quad \mbox{on} \quad (0,1),
	\end{array}
	\right.
\end{equation}
where the initial data $Z^{\varepsilon}$ is computed from the initial data 
$\Phi^{\varepsilon}$: 
\begin{eqnarray}
 Z^{\varepsilon}(x) & = & 
 \exp\left( - \frac{1}{2} \int_0^x \Phi^{\varepsilon}(s) ds \right) \nonumber \\
 & = & \frac{1+e^{\frac{H}{\varepsilon}(x-1)}}{1+e^{-\frac{H}{\varepsilon}}}.
 \label{def.Z}
\end{eqnarray}

An important remark is that $\Phi^{\varepsilon} \leq 0$. Thus, 
by the comparison principle from Lemma~\ref{lemma.comparison},
$\phi \leq 0$ on $[0,T]\times[0,1]$
and $z \geq 1$ on $[0,T]\times[0,1]$. The backwards Cole-Hopf transform will
give us:
\begin{equation*}
 \phi(T) = - 2 \frac{z_x(T)}{z(T)}.
\end{equation*}
Hence, using the fact that $z \geq 1$, we will have the following estimate: 
\begin{equation}
  \label{majo_retour}
  \left| \phi(T, \cdot) \right| \leq 2 \left| z_x(T, \cdot) \right|.
\end{equation}
All we have to do is to study the $L^2$-norm of $z_x(T)$.
To ease computations, let us introduce:
\begin{equation}
\label{relation.zw}
w = (1+e^{-\frac{H}{\varepsilon}}) z_x,
\end{equation}
such that $w$ is the solution to:
\begin{equation*}
	\left\{
	\begin{array}{rll}
		w_t - w_{xx} & = 0
		& \quad \mbox{on} \quad (0,T) \times (0,1), \\
		w(t, 0) &= 0 &	\quad \mbox{on} \quad (0,T), \\
		w(t, 1) &= 0 & \quad \mbox{on} \quad (0,T), \\
		w(0, x) &= \frac{H}{\varepsilon} 
			e^{\frac{H}{\varepsilon}(x-1)} & \quad \mbox{on} \quad (0,1).
	\end{array}
	\right.
\end{equation*}

% ==============================================================================
\subsection{Fourier series decomposition}

We use Fourier series to compute $w(T, \cdot)$. We will use the following 
Hilbert basis of $L^2$ made of the eigen-functions for the Laplace operator 
with Dirichlet boundary conditions on $[0, 1]$:
\begin{equation*}
 e_n(x) = \sqrt{2} \sin (n\pi x) \quad \textrm{for} \quad
 n \geq 1.
\end{equation*}
Let us compute the decomposition of $w(0, \cdot)$ on this basis.
We integrate by parts twice: 
\begin{eqnarray*}
 \langle w(0, \cdot) | e_n \rangle 
 & = & \sqrt{2} \frac{H}{\varepsilon} e^{-\frac{H}{\varepsilon}}
 \int_0^1 \sin (n\pi x) e^{\frac{H}{\varepsilon}x} dx \\
 & = & \sqrt{2} e^{-\frac{H}{\varepsilon}} 
 \left[ \sin (n\pi x) e^{\frac{H}{\varepsilon}x} \right]_0^1
 - \sqrt{2} n \pi e^{-\frac{H}{\varepsilon}} \int_0^1 \cos (n \pi x) e^{\frac{H}{\varepsilon}x} \\
 & = & - \frac{\varepsilon \sqrt{2}}{H} n \pi e^{-\frac{H}{\varepsilon}} \left[ \cos (n\pi x) e^{\frac{H}{\varepsilon}x} \right]_0^1
 - \left(\frac{\varepsilon n \pi}{H}\right)^2 
 \langle w(0, \cdot) | e_n \rangle \\
 & = & \frac{\sqrt{2}}{H}\frac{\varepsilon n \pi}{1 + \frac{\varepsilon^2 n^2 \pi^2}{H^2}} 
 \left( (-1)^{n+1} + e^{-\frac{H}{\varepsilon}} \right).
\end{eqnarray*}
Now we can estimate the size of $w(T, \cdot)$ in $L^2(0,1)$:
\begin{eqnarray*}
 \left\|w(T, \cdot)\right\|_{L^2}^2 & = &
 \sum_{n\geq 1} \left(\langle w(0, \cdot) | e_n \rangle \cdot
 	e^{-n^2 \pi^2 T}\right)^2 \\
 & \leq & 8 \sum_{n\geq 1} \frac{\varepsilon^2 n^2 \pi^2 H^{-2}}
 {\left( 1 + \varepsilon^2 n^2 \pi^2 H^{-2} \right)^2} e^{-2 n^2 \pi^2 T}. \\
\end{eqnarray*}
For $\alpha \in \R$, the following easy inequality holds:
\begin{equation*}
 \frac{\alpha^2}{(1+\alpha^2)^2} \leq \min\left( \alpha^2, \frac{1}{4}\right).
\end{equation*}
Hence we split the sum and cut at a level $N(\varepsilon)$:
\begin{eqnarray*}
 \left\|w(T, \cdot)\right\|_{L^2}^2
 & \leq & 8 \sum_{n = 1}^{N-1} \frac{\varepsilon^2 n^2 \pi^2}{H^2}
 + 2 \sum_{k\geq 0} e^{-2(N+k)^2 \pi^2 T} \\
 & \leq & \frac{8 \varepsilon^2 N^3 \pi^2}{3 H^2}
 + 2 e^{-2 N^2 \pi^2 T} \sum_{k\geq 0} e^{-4 N k \pi^2 T} \\
 & \leq & \frac{8 \varepsilon^2 N^3 \pi^2}{3 H^2}
 + 2 \frac{e^{-2 N^2 \pi^2 T}}{1 - e^{-4 N \pi^2 T}}.
\end{eqnarray*}

We want to choose $N(\varepsilon) \rightarrow +\infty$ such that
$\varepsilon^2 N^3 \rightarrow 0$. For instance, we can take 
$N = \lfloor \varepsilon^{-\eta} \rfloor$, where $\eta > 0$ is small enough.
For $\varepsilon$ small enough, we have:
\begin{equation}
\label{majo_w}
 \left\|w(T, \cdot)\right\|_{L^2}^2
   \leq \frac{8 \pi^2}{3 H^2} \varepsilon^{2-3\eta}
 + 4 e^{-2 \varepsilon^{-2\eta} \pi^2 T} 
 = \mathcal{O}\left(\varepsilon^{2-3\eta}\right). 
\end{equation}
Combining estimates~(\ref{majo_w}) and~(\ref{majo_retour}), and the
definition~(\ref{relation.zw}) we can easily deduce the estimate
(\ref{majo.findissolution}). This concludes the proof of Lemma
\ref{lemma.dissolution}. \qed

% ==============================================================================
\subsection{Approximate controllability towards the null state}

First, let us prove the following technical lemma. Indeed, we have proven that the
particular boundary layer $\Phi^{\varepsilon}$ dissipates, but all we also want
to know what would happen if we were very close to it.

\begin{lemma}
\label{lemma.dissolution2}
Let us change the initial data from Lemma~\ref{lemma.dissolution} to
$\Phi^{\varepsilon}(x) + \frac{1}{\varepsilon} \delta^\varepsilon$. We assume:
\begin{eqnarray}
\label{hyp.negative}
\Phi^{\varepsilon}(x) + \frac{1}{\varepsilon} \delta^\varepsilon
& \leq & 0, \\
\left\| \delta^\varepsilon(\cdot) \right\|_{L^2(0,1)} & = &
\mathcal{O}_{\varepsilon \rightarrow 0}(\varepsilon^3). 
\label{condition.delta}
\end{eqnarray}
Then, the conclusion of Lemma~\ref{lemma.dissolution} still holds.
\end{lemma}

\begin{proof}
We follow the same scheme than for the proof of Lemma~
\ref{lemma.dissolution}. Hence, we start by taking the Cole-Hopf transform of 
the new initial data 
$\Phi^{\varepsilon}(x) + \frac{1}{\varepsilon} \delta^\varepsilon$. 
Therefore, after the Cole-Hopf transform we have the following initial data:
\begin{equation*}
  Z^{\varepsilon}(x) + Z^{\varepsilon}(x) \cdot 
  \left(\exp\left(-\frac{1}{2\varepsilon}\int_0^x\delta^\varepsilon\right)-1\right).
\end{equation*}
From our previous computation~(\ref{def.Z}) of $Z^{\varepsilon}$, we know that
$| (Z^{\varepsilon})_x | = \mathcal{O}(1/\varepsilon)$. Hence, using condition
(\ref{condition.delta}), we have:
\begin{equation*}
\left\|
Z^{\varepsilon}(x) \cdot 
  \left(\exp\left(-\frac{1}{2\varepsilon}\int_0^x\delta^\varepsilon\right)-1\right)
\right\|_{H^1(0,1)} = \mathcal{O}_{\varepsilon \rightarrow 0}(\varepsilon).
\end{equation*}
Let us use the fact that our heat system~(\ref{system.hc}) is linear. 
Therefore, using the conclusion of Lemma~\ref{lemma.dissolution} we have:
\begin{equation*}
\| z(T, \cdot) \|_{H^1(0,1)} = 
\mathcal{O}_{\varepsilon \rightarrow 0}(\varepsilon^{1-\delta})
+ \mathcal{O}_{\varepsilon \rightarrow 0}(\varepsilon).
\end{equation*}
Once again we apply the backwards Cole-Hopf transform. We use the fact that
$z \geq 1$ (this comes from the comparison principle and the hypothesis 
(\ref{hyp.negative})). Hence, 
\begin{equation*}
\| \phi(T, \cdot) \|_{L^2(0,1)} \leq 2 \| z(T, \cdot) \|_{H^1(0,1)}.
\end{equation*}
Thus, the conclusion~(\ref{majo.findissolution}) of Lemma~
\ref{lemma.dissolution} still holds with this new initial data.
\end{proof}

Now everything is ready for us to show the following small time approximate
controllability result for system~(\ref{burgers2}). We have to combine the 
different estimates.

\begin{theorem}
\label{thm.frederic2}
Let $T, r > 0$ and $y_0 \in L^2(0,1)$ be given data. Then there exists 
$u, v \in \EnsU \times \EnsV$ such that the associated solution $y \in X$ to 
system~(\ref{burgers2}) on $[0, T]$ satisfies:
\begin{equation*}
 \left\| y(T, \cdot) \right\|_{L^2(0,1)} \leq r.
\end{equation*}
\end{theorem}

\begin{proof}
Take $T, r > 0$ and $y_0 \in L^2(0,1)$ given data.  
Let us take a small $\varepsilon > 0$ and break down
our time interval into four parts. We introduce $T_1 = T/3$, 
$T_2 = T_1 + \varepsilon$ and $T_3 = T_2 + \varepsilon^4$. 
The first part $[0,T_1]$ of length $T/3$ is
designed to smooth the initial data. The second part $[T_1,T_2]$ of length
$\varepsilon$ is the part where the settling of the
boundary layer takes place. The third part $[T_2, T_3]$ of length
$\varepsilon^4$ is the quick push down to zero. The fourth part $[T_3, T]$
of length at least $T/3$ (when $\varepsilon$ is small enough) is the passive
stage for the dissipation of the boundary layer. Let us give some details.

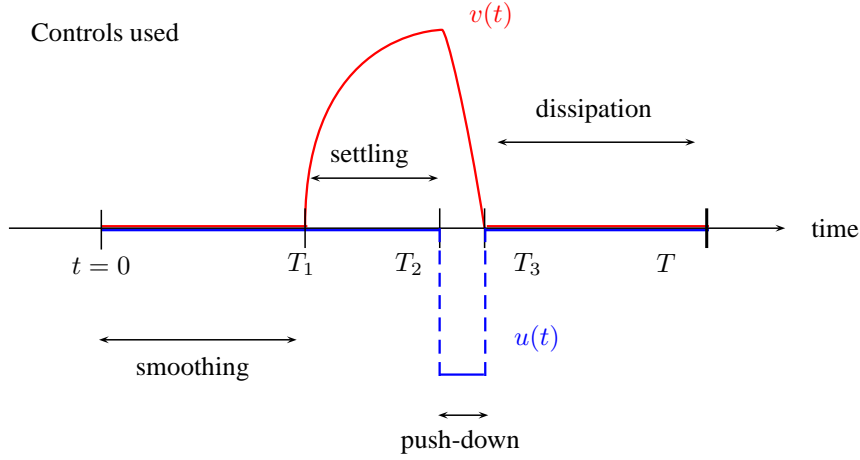
\begin{figure}[!ht]
\centering
% Generated with LaTeXDraw 2.0.8
% Tue Dec 18 15:27:40 CET 2012
% \usepackage[usenames,dvipsnames]{pstricks}
% \usepackage{epsfig}
% \usepackage{pst-grad} % For gradients
% \usepackage{pst-plot} % For axes
\scalebox{1} % Change this value to rescale the drawing.
{
\begin{pspicture}(0,-3.0442185)(11.30297,3.0442185)
\psbezier[linewidth=0.03,linecolor=red](3.8951595,0.04694975)(3.8951595,2.5984044)(5.6361775,2.651012)(5.695195,2.651012)(5.7542124,2.651012)(6.019792,1.5462581)(6.255862,0.0206461)
\psline[linewidth=0.03cm,linecolor=red](3.9246683,0.04694975)(1.2098601,0.04694975)
\psline[linewidth=0.03cm,linecolor=red](6.285371,0.04694975)(9.20674,0.04694975)
\psline[linewidth=0.03cm,linecolor=blue](1.2098601,-0.00565758)(5.6656866,-0.00565758)
\psline[linewidth=0.03cm,linecolor=blue](5.6561775,-1.9206097)(6.246353,-1.9206097)
\psline[linewidth=0.03cm,linecolor=blue](6.255862,-0.00565758)(9.20674,-0.00565758)
\psline[linewidth=0.04cm](9.177232,-0.24239045)(9.177232,0.28368264)
\psline[linewidth=0.02cm](6.255862,-0.24239045)(6.255862,0.28368264)
\psline[linewidth=0.02cm](5.6656866,-0.24239045)(5.6656866,0.28368264)
\psline[linewidth=0.02cm](3.8951595,-0.24239045)(3.8951595,0.28368264)
\psline[linewidth=0.02cm,arrowsize=0.05291667cm 2.0,arrowlength=1.4,arrowinset=0.4]{->}(0.0,0.0206461)(10.210039,0.0206461)
\usefont{T1}{ptm}{m}{n}
\rput(6.9435935,-1.4392189){\color{blue}$u(t)$}
\usefont{T1}{ptm}{m}{n}
\rput(6.3435936,2.8407812){\color{red}$v(t)$}
\usefont{T1}{ptm}{m}{n}
\rput(3.8435936,-0.45921874){$T_1$}
\usefont{T1}{ptm}{m}{n}
\rput(5.263593,-0.45921874){$T_2$}
\usefont{T1}{ptm}{m}{n}
\rput(6.8235936,-0.45921874){$T_3$}
\usefont{T1}{ptm}{m}{n}
\rput(8.643596,-0.43921876){$T$}
\psline[linewidth=0.03cm,linecolor=blue,linestyle=dashed,dash=0.17638889cm 0.10583334cm](5.664531,0.02078135)(5.664531,-1.9392186)
\psline[linewidth=0.03cm,linecolor=blue,linestyle=dashed,dash=0.17638889cm 0.10583334cm](6.264531,0.02078135)(6.264531,-1.9392186)
\usefont{T1}{ptm}{m}{n}
\rput(10.867032,0.02078125){time}
\usefont{T1}{ptm}{m}{n}
\rput(1.2617195,2.6207812){Controls used}
\psline[linewidth=0.02cm](1.2151595,-0.26239046)(1.2151595,0.26368266)
\usefont{T1}{ptm}{m}{n}
\rput(1.2135936,-0.47921875){$t = 0$}
\psline[linewidth=0.02cm,arrowsize=0.05291667cm 2.0,arrowlength=1.4,arrowinset=0.4]{<->}(1.1845312,-1.4592186)(3.784531,-1.4592186)
\usefont{T1}{ptm}{m}{n}
\rput(2.4110944,-1.8392187){smoothing}
\psline[linewidth=0.02cm,arrowsize=0.05291667cm 2.0,arrowlength=1.4,arrowinset=0.4]{<->}(6.424531,1.1607814)(9.024531,1.1607814)
\usefont{T1}{ptm}{m}{n}
\rput(7.692032,1.5807812){dissipation}
\psline[linewidth=0.02cm,arrowsize=0.05291667cm 2.0,arrowlength=1.4,arrowinset=0.4]{<->}(3.9645312,0.68421865)(5.6445312,0.68421865)
\usefont{T1}{ptm}{m}{n}
\rput(4.741094,0.98078126){settling}
\psline[linewidth=0.02cm,arrowsize=0.05291667cm 2.0,arrowlength=1.4,arrowinset=0.4]{<->}(5.6445312,-2.4592185)(6.304531,-2.4592185)
\usefont{T1}{ptm}{m}{n}
\rput(5.9292197,-2.8192186){push-down}
\end{pspicture} 
}
\caption{Approximate null-controllability strategy.}
\end{figure}

\skipline

\textbf{Smoothing of the initial data:}\quad
First, for $t \in [0, T_1]$, we choose $u(t) = v(t) = 0$. The system evolves 
freely. Regularization effects of the Burgers' equation smooth our initial 
data $y_0 \in L^2(0,1)$. We have $y(T_1, \cdot) \in H^1_0(0,1)$. There are many 
ways to prove such a result. For instance, one can take the Cole-Hopf transform 
and use well-known regularization properties of the heat equation.

\skipline

\textbf{Settling of the boundary layer:}\quad
Next, for $t \in [T_1, T_2]$, we perform the scaling~(\ref{scaling}). We want to
apply Lemma~\ref{lemma.settling} for a duration $1$. Hence, let us choose some 
$H$ such that $H - 2 > 0$. We take $\bar{v} \in H^{3/4}(0,T)$ the control from 
Lemma~\ref{lemma.settling}. For $t \in [T_1, T_2]$, we use:
\begin{eqnarray*}
 u(t) & = & 0, \\ 
 v(t) & = & \frac{1}{\varepsilon} \bar{v}\left(\frac{t-T_1}{\varepsilon}\right).
\end{eqnarray*}
From Lemma~\ref{lemma.settling}, we know that:
\begin{equation}
\left\| y(T_2, \cdot) - \frac{1}{\varepsilon}h_{\varepsilon}(\cdot) \right\|_{L^2(0,1)}
= \mathcal{O}_{\varepsilon \rightarrow 0} 
\left( \varepsilon^{-3/2} e^{-\frac{H}{4\varepsilon}\left(H - 2\right)} \right).
\label{estimate.finalproof}
\end{equation}

\skipline

\textbf{Push-down towards zero:}\quad
Then, still in the context of scaling~(\ref{scaling}), we want to apply 
Lemma~\ref{lemma.push} during a very short duration $\varepsilon^3$. Hence, for 
$t \in [T_2, T_3]$, we choose the controls found 
in Lemma~\ref{lemma.push} (with a total time $\varepsilon^3$), and we scale them
appropriately. That is to say:
\begin{eqnarray*}
 u(t) & = & \frac{1}{\varepsilon^2} \bar{u}\left(\frac{t-T_2}{\varepsilon}\right), \\ 
 v(t) & = & \frac{1}{\varepsilon} \bar{v}\left(\frac{t-T_2}{\varepsilon}\right).
\end{eqnarray*}
Combining (\ref{estimate.finalproof}) and Lemma~
\ref{lemma.push}, we get that, at the end of this hyperbolic stage:
\begin{equation*}
0 \geq y(T_3, \cdot) 
\geq \Phi^\varepsilon + \frac{1}{\varepsilon} \delta(\varepsilon^3, \cdot) 
- \frac{1}{\varepsilon} \left\| h^\varepsilon - k^\varepsilon \right\|_{\infty},
\end{equation*}
where (using estimate~(\ref{estimate.steady})):
\begin{equation*}
\left\| \delta(\varepsilon^3, \cdot) \right\|_{L^2} 
+ \left\| h^\varepsilon - k^\varepsilon \right\|_{\infty}  = 
\mathcal{O}_{\varepsilon \rightarrow 0}(\varepsilon^3).
\end{equation*}

\skipline

\textbf{Dissipation of the boundary residue:}\quad
Now we enter the passive stage. We choose $v(t) = u(t) = 0$ for $t\in [T_3, T]$.
Since $\varepsilon$ goes to zero, $T - T_3 \geq T/3$. Hence we can apply
Lemma~\ref{lemma.dissolution2} on a time interval independent of $\varepsilon$.
By using the comparison principle from Lemma~\ref{lemma.comparison} we can 
conclude that:
\begin{equation*}
\left\| y(T, \cdot) \right\|_{L^2} = 
\mathcal{O}_{\varepsilon \rightarrow 0}(\varepsilon^{1-\eta}),
\end{equation*}
for any $\eta > 0$. For instance, one can choose $\eta = \frac{1}{2}$. Then we 
choose $\varepsilon$ small enough to ensure that
$\left\| y(T, \cdot) \right\|_{L^2} \leq r$.
This concludes the proof of Theorem~\ref{thm.frederic2}.
\end{proof}

\begin{remark}
In the proof of Theorem~\ref{thm.frederic2}, we concatenate different
controls found in different parts. This could be a problem for smoothness 
because we did not check compatibility conditions at the jointures. However,
the proof provides a control $v \in H^{1/4}(0,T)$ and this doesn't require
compatibility conditions. If one wants smooth controls, it is also 
possible. One can choose a smooth control close to our control for the 
approximate controllability, then end with a smooth control for the exact
controllability.
\end{remark}

% ==============================================================================
\section{Parabolic stage and exact local controllability}
\label{section.parabolic}
% ==============================================================================

Theorem~\ref{thm.frederic2} takes care of the small time global approximate 
controllability towards the null state. To get Theorem~\ref{thm.frederic},
we need to combine it with a small time local exact controllability result 
in the vicinity of  the null state. We give in this section two different
approaches for this type of result.

% ==============================================================================
\subsection{Fursikov and Imanuvilov's theorem}

The following theorem is due to Fursikov and Imanuvilov. Indeed, the techniques
they expose in their book~\cite{MR1406566} can be applied to show the following
result. However, the proof of this precise statement is not written, and one has
to work to show that the control can be chosen to be smooth.

\begin{theorem}
\label{thm.fursikov}
Let $T > 0$. There exists $r > 0$ such that, for any initial data 
$y_0 \in L^2(0,1)$ satisfying:
\begin{equation}
\label{petitesse}
\| y_0 \|_{L^2(0,1)} \leq r, 
\end{equation}
there exists a control $v \in \mathcal{C}^1[0,T]$ such that the solution 
$y \in X$ to the system:
\begin{equation}
\label{burgers_fursikov}
	\left\{
	\begin{array}{rll}
		y_t + y y_x - y_{xx} & = 0
		& \quad \mbox{in} \quad (0,T) \times (0,1), \\
		y(t, 0) & = v(t) &	\quad \mbox{in} \quad (0,T), \\
		y(t, 1) & = 0 &	\quad \mbox{in} \quad (0,T), \\
		y(0, x) & = y_0(x) & \quad \mbox{in} \quad (0,1),
	\end{array}
	\right.
\end{equation}
satisfies $y(T, \cdot) \equiv 0$.
\end{theorem}

The full theorem is in fact more general since one obtains local exact 
controllability to the trajectories of system~(\ref{burgers_fursikov}). 
The proof relies on Carleman estimates for 
parabolic equations. It is an extension of a previous result with two boundary
controls whose proof can be read in~\cite{MR1348646}.

% ==============================================================================
\subsection{Using Cole-Hopf and a moments method}

In this section we give a proof of Theorem~\ref{thm.fursikov} (both 
for the sake of completeness and for avoiding Carleman estimates). It relies on 
the Cole-Hopf transform and a moments method introduced in~\cite{MR0335014} by 
Fattorini and Russell.

\begin{proof}
Let $T > 0$. First, we consider the following heat system:
\begin{equation}
\label{syst.chaleur}
	\left\{
	\begin{array}{rll}
		z_t - z_{xx} & = 0
		& \quad \mbox{in} \quad (0,T) \times (0,1), \\
		z(t, 0) & = \alpha(t) & \quad \mbox{in} \quad (0,T), \\
		z_x(t, 1) & = 0 & \quad \mbox{in} \quad (0,T), \\
		z(0, x) & = z_0(x) & \quad \mbox{in} \quad (0,1).
	\end{array}
	\right.
\end{equation}
This is typically a setting for which we can apply the moments method of
Fattorini and Russel exposed in~\cite{MR0335014}. They prove this system is null
controllable for any positive time by means of very smooth controls. Let us use
some control $\alpha \in \mathcal{C}^1[0,T]$. They also prove that there exists 
some constant $C_T$ such that the size of the control is bounded from above by
$C_T \times \| z_0 \|_{L^2}$. Therefore, if $z_0$ is small enough in $L^2$,
one can steer it to zero with a control $\alpha(\cdot)$ such that 
$\|\alpha(\cdot)\|_{\infty} < 1$.

Now we get back to our Burgers' system. For $y_0 \in L^2(0,1)$, let us choose:
\begin{equation*}
z_0(x) = \exp\left(-\frac{1}{2}\int_0^1 y_0(s) ds\right) - 1.
\end{equation*}
Thus, if $y_0(\cdot)$ is small in $L^2(0,1)$ then $z_0(\cdot)$ too. If they are 
small enough, then we can steer $z_0$ to $0$ with a control such that 
$\|\alpha(\cdot)\|_{\infty} < 1$. In that setting, we have $z(\cdot) > -1$
thanks to the maximum principle for the heat equation.
Hence, if we let $y = - 2 z_x / (1 + z)$, we get a solution $y \in X$ to 
(\ref{burgers_fursikov}) such that $y(T, \cdot) \equiv 0$ provided that 
condition~(\ref{petitesse}) is satisfied for some $r > 0$ depending only on $T$.
\end{proof}

% ==============================================================================
\section{Conclusion}
% ==============================================================================

In our work, we want to underline two important ideas. The first one is the
rigorous analysis of the hyperbolic limit system and of the adequate weak 
boundary conditions. These weak boundary conditions somehow describe the
behavior of the boundary layer and what it will be able to do or not.
The second idea is the dissipation of the boundary layer by the fluid system
itself during the passive stage. Once a boundary layer is created, will the 
system be able to dissipate it in short time or not? 

These two ideas might be important for the analysis of more complex problems
such as the Navier-Stokes Open problem \ref{open.NS}. 
For instance, one could try to see if the boundary layer appearing in 
\cite{MR1393067} when trying to control the 2D Navier-Stokes system with Navier 
slip boundary conditions can be dissipated in small time by the system itself.

\skipline

The author would like to thank his advisor Jean-Michel Coron for having 
attracted his attention on this control problem, 
Claude Bardos, Sergio Guerrero, for fruitful discussions 
and Vincent Perrollaz for his advice concerning the hyperbolic system.

% ==============================================================================

\bibliographystyle{plain}
\bibliography{Biblio-Burgers-BLC}

% ==============================================================================

\end{document}